\documentclass[11pt]{amsart}
\usepackage{a4wide,amssymb,color}
\usepackage{relsize}
\usepackage[all]{xy}
\usepackage{bbm}

\newcommand*{\defeq}{\stackrel{\mathsmaller{\mathsf{def}}}{=}}
\newcommand{\w}{\omega}

\newcommand{\IN}{\mathbb N}

\newcommand{\two}{\mathbbm 2}
\newcommand{\IZ}{\mathbb Z}

\newtheorem{theorem}{Theorem}[section]
\newtheorem{proposition}[theorem]{Proposition}

\newtheorem{corollary}[theorem]{Corollary}

\theoremstyle{definition}

\newtheorem{remark}[theorem]{Remark}

\title{The binary quasiorder on  semigroups}
\author{Taras Banakh and Olena Hryniv}
\address{Ivan Franko National University of Lviv, Universytetska 1, 79000, Lviv, Ukraine}
\email{t.o.banakh@gmail.com, ohryniv@gmail.com}
\subjclass{20M10}
\keywords{$E$-central semigroup, the least semilattice congruence, the binary quasiorder}

\begin{document}
\begin{abstract} Given two elements $x,y$ of a semigroup $X$ we write $x\lesssim y$ if for every homomorphism $\chi:X\to\{0,1\}$ we have $\chi(x)\le\chi(y)$. The quasiorder $\lesssim$ is called the {\em binary quasiorder} on $X$. It induces the equivalence relation $\Updownarrow$ that coincides with the least semilattice congruence on $X$. In the paper we discuss some known and new properties of the binary quasiorder on semigroups.
\end{abstract}



\maketitle

\section{Introduction}

In this paper we study the binary quasiorder on semigroups. 
 Every semigroup carries many important quasiorders (for example, those generated by the Green relations). One of them is the binary quasiorder $\lesssim$ defined as follows. Given two elements $x,y$ of a semigroup $X$ we write $x\lesssim y$ if $\chi(x)\le\chi(y)$ for any homomorphism $\chi:X\to\{0,1\}$. On every semigroup $X$ the binary quasiorder generates a congruence, which coincides with the least semilattice congruence, and decomposes the semigroup into a semilattice of semilattice-indecomposable semigroups. This fundamental decomposition result was proved by Tamura \cite{Tam56} (see also \cite{Petrich63}, \cite{Petrich64}, \cite{TS66}). Because of its fundamental importance, the least semilattice congruence has been deeply studied by many mathematicians, see the papers \cite{CB93}, \cite{CB}, \cite{Galbiati}, \cite{Gigon}, \cite{MRV}, \cite{PBC}, \cite{Putcha73}, \cite{Putcha74}, \cite{PW}, \cite{Petrich63}, \cite{Petrich64}, \cite{Sulka}, \cite{TK54}, \cite{Tamura73}, \cite{Tamura82}, surveys \cite{Mitro2004}, \cite{MS}, and monographs \cite{BCP}, \cite{Mitro2003},  \cite{Petrich73}.  
 The aim of this paper is to provide a survey of  known and new results on the binary quasiorder and the least semilattice congruence on semigroups. The obtained results will be applied in the theory of categorically closed semigroups developed by the first author in collaboration with Serhii Bardyla, see \cite{BB1,BB2,BB3,BB4,BB5}.

\section{Preliminaries}

In this section we collect some standard notions that will be used in the paper. We refer to \cite{Howie} for Fundamentals of Semigroup Theory. 

We denote by $\w$  the set of all finite ordinals and by $\IN\defeq\w\setminus\{0\}$ the set of all positive integer numbers.

 A {\em semigroup} is a set endowed with an associative binary operation. A semigroup $X$ is called a {\em semilattice}  if $X$ is commutative and every element $x\in X$ is an {\em idempotent} which means $xx=x$. Each semilattice $X$ carries the {\em natural partial order} $\leq$ defined by $x\le y$ iff $xy=x$.
 For a semigroup $X$ we denote by $E(X)\defeq\{x\in X:xx=x\}$ the set of idempotents of $X$.


Let $X$ be a semigroup. For an element $x\in X$ let $$x^\IN\defeq\{x^n:n\in\IN\}$$ be the monogenic subsemigroup of $X$, generated by the element $x$. For two subsets $A,B\subseteq X$, let $AB\defeq\{ab:a\in A,\;b\in B\}$ be the product of $A,B$ in $X$. 

For an element $a$ of a semigroup $X$, the set
$$H_a=\{x\in X:(xX^1=aX^1)\;\wedge\;(X^1x=X^1a)\}$$
is called the {\em $\mathcal H$-class} of $a$.
Here $X^1=X\cup\{1\}$ where $1$ is an element such that $1x=x=x1$ for all $x\in X^1$. By Corollary 2.2.6 \cite{Howie}, for every idempotent $e\in E(X)$ its $\mathcal H$-class $H_e$ coincides with the maximal subgroup of $X$, containing the idempotent $e$. 

\section{The binary quasiorder}

In this section we discuss the binary quasiorder on a semigroup and its relation to the least semilattice congruence. 

Let $\two$ denote the set $\{0,1\}$ endowed with the operation of multiplication inherited from the ring $\IZ$. It is clear that $\two$ is a two-element semilattice, so it carries the natural partial order, which coincides with the linear order inherited from $\IZ$.

For elements $x,y$ of a semigroup $X$ we write $x\lesssim y$ if $\chi(x)\le \chi(y)$ for every homomorphism $\chi:X\to\two$. It is clear that $\lesssim$ is a quasiorder on $X$. This quasiorder will be referred to as {\em the binary quasiorder} on $X$. The obvious order properties of the semilattice $\two$ imply the following (obvious) properties of the binary quasiorder on $X$.

\begin{proposition}\label{p:quasi2} For any semigroup $X$ and any elements $x,y,a\in X$, the following statements hold:
\begin{enumerate}
\item if $x\lesssim y$, then $ax\lesssim ay$ and $xa\lesssim ya$;
\item $xy\lesssim yx\lesssim xy$;
\item $x\lesssim x^2\lesssim x$;
\item $xy\lesssim x$ and $xy\lesssim y$.
\end{enumerate}
\end{proposition}

For an element $a$ of a semigroup $X$ and subset $A\subseteq X$, consider the following sets:
$$
{\Uparrow}a\defeq\{x\in X:a\lesssim x\},\quad {\Downarrow}a\defeq\{x\in X:x\lesssim a\},\quad\mbox{and}\quad{\Updownarrow}a\defeq\{x\in X:a\lesssim x\lesssim a\},
$$
called the {\em upper $\two$-class}, {\em lower $\two$-class} and the {\em $\two$-class} of $x$, respectively. Proposition~\ref{p:quasi2} implies that those three classes are subsemigroups of $X$. 
 
For two elements $x,y\in X$ we write $x\Updownarrow y$ iff ${\Updownarrow}x={\Updownarrow}y$ iff $\chi(x)=\chi(y)$ for any homomorphism $\chi:X\to\two$. Proposition~\ref{p:quasi2} implies that $\Updownarrow$ is a congruence on $X$. 

 We recall that a {\em congruence} on a semigroup $X$ is an equivalence relation $\approx$ on $X$ such that for any elements $x\approx y$ of $X$ and any $a\in X$ we have $ax\approx ay$ and $xa\approx ya$. For any congruence $\approx$ on a semigroup $X$, the quotient set $X/_\approx$ has a unique semigroup structure such that the quotient map $X\to X/_\approx$ is a semigroup homomorphism. The semigroup $X/_\approx$ is called the {\em quotient semigroup} of $X$ by the congruence $\approx$~.

A congruence $\approx$ on a semigroup $X$ is called a {\em semilattice congruence} if the quotient semigroup $X/_\approx$ is a semilattice. Proposition~\ref{p:quasi2} implies that $\Updownarrow$ is a semilattice congruence on $X$. The intersection of all semilattice congruences on a semigroup $X$ is a semilattice congruence called the {\em least semilattice congruence},  denoted by $\eta$ in \cite{Howie}, \cite{HL}  (by $\xi$ in \cite{Tamura73}, \cite{Mitro2004}, and by $\rho_0$ in \cite{BCP}). The minimality of $\eta$ implies that $\eta\subseteq {\Updownarrow}$. The inverse inclusion ${\Updownarrow}\subseteq\eta$  will be deduced from the following (probably known) theorem on extensions of $\two$-valued  homomorphisms.

\begin{theorem}\label{t:extend} Let $\pi:X\to Y$ be a surjective homomorphism from a semigroup $X$ to a semilattice $Y$. For every subsemilattice $S\subseteq Y$ and homomorphism $f:\pi^{-1}[S]\to\two$ there exists a homomorphism $F:X\to\two$ such that $F{\restriction}_{\pi^{-1}[S]}=f$.
\end{theorem}

\begin{proof}  
We claim that the function
$F:X\to\two$ defined by
$$F(x)=\begin{cases}1&\mbox{if $\exists z\in\pi^{-1}[S]$ such that $\pi(xz)\in S$ and $f(xz)=1$};\\
0&\mbox{otherwise};
\end{cases}
$$is a required homomorphism extending $f$. 

To see that $F$ extends $f$, take any $x\in\pi^{-1}[S]$. If $f(x)=1$, then for $z=x$ we have $\pi(xz)=\pi(x)\pi(z)=\pi(x)\pi(x)=\pi(x)\in S$ and $f(xz)=f(x)f(z)=f(x)f(x)=1$ and hence $F(x)=1=f(x)$. If $F(x)=1$, then there exists $z\in \pi^{-1}[S]$ such that $\pi(xz)\in S$ and $f(x)f(z)=f(xz)=1$, which implies that $f(x)=1$. Therefore, $F(x)=1$ if and only if $f(x)=1$. Since $\two$ has only two elements, this implies that $f=F{\restriction}_{\pi^{-1}[S]}$.

To show that $F$ is a homomorphism, fix any elements $x_1,x_2\in X$.
We should prove that $F(x_1x_2)=F(x_1)F(x_2)$. 

First assume that $F(x_1x_2)=0$. If $F(x_1)$ or $F(x_2)$ equals $0$, then $F(x_1)F(x_2)=0$ and we are done. So, assume that $F(x_1)=1=F(x_2)$. Then the definition of $F$ yields elements $z_1,z_2\in \pi^{-1}[S]$ such that $\pi(x_iz_i)\in S$ and $f(x_iz_i)=1$ for every $i\in\{1,2\}$. Now consider the element $z=z_1z_2\in\pi^{-1}[S]$ and observe that $$\pi(x_1x_2z)=\pi(x_1x_2z_1z_2)=\pi(x_1)\pi(x_2)\pi(z_1)\pi(z_2)=\pi(x_1z_1)\pi(x_2z_2)\in S$$
and $f(z)=f(z_1z_2)=f(z_1)f(z_2)=1\cdot 1=1$ and hence $F(x_1x_2)=1$ by the definition of $F$. By this contradicts our assumption. 

Next, assume that $F(x_1x_2)=1$. Then there exists $z\in\pi^{-1}[S]$ such that $\pi(x_1x_2z)\in S$ and $f(x_1x_2z)=1$. Let $z'=x_1x_2z\in\pi^{-1}[S]$ and observe that for every $i\in\{1,2\}$ we have $\pi(x_iz')=\pi(x_i)\pi(x_1)\pi(x_2)\pi(z)=\pi(x_1)\pi(x_2)\pi(z)=\pi(x_1x_2z)\in S$.  It follows from $1=f(x_1x_2z)=f(x_1)f(x_2)f(z)=f(x_i)f(x_1)f(x_2)f(z)$ that $f(x_i)=1=f(z')$ and hence $F(x_i)=1$. Then $F(x_1)F(x_2)=1=F(x_1x_2)$, which completes the proof.
\end{proof}

\begin{corollary}\label{c:extend} Any homomorphism $f:S\to \two$ defined on a subsemilattice $S$ of a semilattice $X$ can be extended to a homomorphism $F:X\to\two$.
\end{corollary}

\begin{proof} Apply Theorem~\ref{t:extend} to the identity homomorphism $\pi:X\to X$.
\end{proof}

Corollary~\ref{c:extend} implies the following important fact, first noticed by Petrich \cite{Petrich63}, \cite{Petrich64} and Tamura \cite{Tamura73}. 

\begin{theorem} The congruence $\Updownarrow$ on any semigroup $X$ coincides with the least semilattice congruence on  $X$.
\end{theorem}

\begin{proof} Let $\eta$ be the least semilattice congruence on $X$ and $\eta(\cdot):X\to X/\eta$ be the quotient homomorphism assigning to each element $x\in X$ its equivalence class $\eta(x)\in X/\eta$. We need to prove that $\eta(x)={\Updownarrow}x$ for all $x\in X$. Taking into account that ${\Updownarrow}$ is a semilattice congruence and $\eta$ is the least semilattice congruence on $X$, we conclude that $\eta\subseteq{\Updownarrow}$ and hence $\eta(x)\subseteq {\Updownarrow}x$ for all $x\in X$. Assuming that $\eta\ne{\Updownarrow}$, we can find elements $x,y\in X$ such that $x\Updownarrow y$ but $\eta(x)\ne \eta(y)$. Consider the subsemilattice $S=\{\eta(x),\eta(y),\eta(x)\eta(y)\}$ of the semilattice $X/\eta$. It follows from $\eta(x)\ne\eta(y)$ that $\eta(x)\eta(x)\ne\eta(x)$ or $\eta(x)\eta(y)\ne\eta(y)$. Replacing the pair $x,y$ by the pair $y,x$, we can assume that $\eta(x)\eta(y)\ne \eta(y)$. In this case the unique function $h:S\to\two$ with $h^{-1}(1)=\{\eta(y)\}$ is a homomorphism. By Corollary~\ref{c:extend}, the homomophism $h$ can be extended to a homomorphism $H:X/\eta\to\two$. Then the composition $\chi\defeq H\circ\eta(\cdot):X\to\two$ is a homomorphism such that $\chi(x)=0\ne 1=\chi(y)$, which implies that ${\Updownarrow}x\ne{\Updownarrow}y$. But this contradicts the choice of the points $x,y$. This contradicton completes the proof of the equality ${\Updownarrow}=\eta$.
\end{proof}

A semigroup $X$ is called {\em $\two$-trivial} if every homomorphism $h:X\to\two$  is constant. Tamura \cite{Tamura73}, \cite{Tamura82} calls $\two$-trivial semigroups {\em semilattice-indecomposable} (or briefy {\em $s$-indecomposable}) semigroups.

Theorem~\ref{t:extend} implies the following fundamental fact first proved by Tamura \cite{Tam56} and then reproved by another method in \cite{TS66}, see also \cite{Petrich63}, \cite{Petrich64}.

\begin{theorem}[Tamura]\label{t:Tamura} For every element $x$ of a semigroup $X$ its $\two$-class ${\Updownarrow}x$ is a $\two$-trivial semigroup.
\end{theorem}

Now we provide an inner description of the binary quasiorder via prime (co)ideals, following the approach of Petrich \cite{Petrich64} and Tamura \cite{Tamura73}.

A subset $I$ of a semigroup $X$ is called
\begin{itemize}
\item an {\em ideal} in $X$ if $(IX)\cup (XI)\subseteq I$;
\item a {\em prime ideal} if $I$ is an ideal such that $X\setminus I$ is a subsemigroup of $X$;
\item a ({\em prime}) {\em coideal} if the complement $X\setminus I$ is a (prime) ideal in $X$.
\end{itemize}
According to this definition, the sets $\emptyset$ and $X$ are prime (co)ideals in $X$.

Observe that a subset $A$ of a semigroup $X$ is a prime coideal in $X$ if and only if its {\em characteristic function} $$\chi_A:X\to\two,\quad \chi_A:x\mapsto\chi_A(x)\defeq\begin{cases}1&\mbox{if $x\in A$},\\
0&\mbox{otherwise},
\end{cases}
$$ is a homomorphism.  This function characterization of prime coideals implies the following inner description of the $\two$-quasiorder, first noticed by Tamura in \cite{Tamura73}.

\begin{proposition}\label{p:smallest-pi} For any element $x$ of a semigroup $X$, its upper $\two$-class ${\Uparrow}x$ coincides with the smallest coideal of $X$ that contains $x$. 
\end{proposition}

The following inner description of the upper $\two$-classes is a modified version of Theorem 3.3 in \cite{Petrich64}.

\begin{proposition}\label{p:Upclass} For any element $x$ of a semigroup $X$ its upper $\two$-class ${\Uparrow}x$ is equal to the union $\bigcup_{n\in\w}{\Uparrow}_{\!n}x$, where ${\Uparrow}_{\!0}x=\{x\}$ and $${\Uparrow}_{\!n{+}1}x\defeq\{y\in X:X^1yX^1\cap({\Uparrow}_{\!n}x)^2\ne \emptyset\}$$ for $n\in\w$. 
\end{proposition}  

\begin{proof} Observe that for every $n\in\w$ and $y\in {\Uparrow}_{\!n}x$ we have $yy\in X^1yX^1\cap({\Uparrow}_{\!n}x)^2\ne\emptyset$ and hence $y\in {\Uparrow}_{\!n{+}1}x$. Therefore, $({\Uparrow}_{\!n}x)_{n\in\w}$ is an increasing sequence of sets. 
Also, for every $y,z\in{\Uparrow}_{\!n}x$ the we have $yz\in X^1yzX^1\cap({\Uparrow}_{\!n}x)^2$ and hence $yz\in {\Uparrow}_{\!n+1}x$, which implies that the union ${\Uparrow}_{\!\w}x\defeq\bigcup_{n\in\w}{\Uparrow}_{\!n}x$ is a subsemigroup of $X$. 

The definition of the sets ${\Uparrow}_{\!n}x$ implies that the complement $I=X\setminus{\Uparrow}_{\!\w}x$ is an ideal in $X$. Then ${\Uparrow}_{\!\w}x$ is a prime coideal in $X$. Taking into account that ${\Uparrow}x$ is the smallest prime coideal containing $x$, we conclude that ${\Uparrow}x\subseteq{\Uparrow}_{\!\w}x$. To prove that ${\Uparrow}_{\!\w}x\subseteq{\Uparrow}x$, it suffices to check that ${\Uparrow}_{\!n}x\subseteq{\Uparrow}x$ for every $n\in\w$. It is trivially true for $n=0$. Assume that for some $n\in\w$ we have already proved that ${\Uparrow}_{\!n}x\subseteq {\Uparrow}x$. Since ${\Uparrow}x$ is a coideal in $X$, for any $y\in X\setminus{\Uparrow}x$ we have $\emptyset=X^1yX^1\cap {\Uparrow}x\supseteq  X^1yX^1\cap{\Uparrow}_{\!n}x$, which implies that $y\notin{\Uparrow}_{\!n{+}1}x$ and hence ${\Uparrow}_{\!n{+}1}\subseteq{\Uparrow}x$. By the Principle of Mathematical Induction, ${\Uparrow}_{\!n}x\subseteq {\Uparrow}x$ for all $n\in\w$ and hence ${\Uparrow}_{\!\w}x=\bigcup_{n\in\w}{\Uparrow}_{\!n}x\subseteq {\Uparrow}x$, and finally ${\Uparrow}_{\!\w}x={\Uparrow}x$.
\end{proof}

For a positive integer $n$, let
$$2^{<n}\defeq\bigcup_{k<n}\{0,1\}^k\quad\mbox{and}\quad 2^{\le n}\defeq\bigcup_{k\le n}\{0,1\}^k.$$For a sequence $s=(s_0,\dots,s_{n-1})\in 2^n$ and a number $k\in\{0,1\}$, let $$s\hat{\;}k\defeq (s_0,\dots,s_{n-1},k)\quad\mbox{and}\quad k\hat{\;}s\defeq (k,s_0,\dots,s_{n-1}).$$
The following proposition provides a constructive description of elements of the sets ${\Uparrow}_{\!n}x$ appearing in Proposition~\ref{p:Upclass}.

\begin{proposition}\label{p:Upclass-tree} For every $n\in\IN$ and every element $x$ of a semigroup $X$, the set ${\Uparrow}_{\!n}x$ coincides with the set ${\Uparrow}_{\!n}'x$ of all elements $y\in X$ for which there exist sequences  $\{x_s\}_{s\in 2^{\le n}}$, $\{y_s\}_{s\in 2^{\le n}}\subseteq X$ and $\{a_{s}\}_{s\in 2^{\le n}},\{b_s\}_{s\in 2^{\le n}}\subseteq X^1$ satisfying the following conditions:
\begin{enumerate}
\item[$(1_n)$] $x_s=x$ for all $s\in 2^n$;
\item[$(2_n)$] $y_s=a_sx_sb_s$ for every $s\in 2^{\le n}$;
\item[$(3_n)$] $y_s=x_{s\hat{\;}0}x_{s\hat{\;}1}$ for every $s\in 2^{<n}$;
\item[$(4_n)$] $x_{()}=y$ for the unique element $()$ of $2^0$.
\end{enumerate}
\end{proposition}

\begin{proof} This proposition will be proved by induction on $n$. For $n=1$,we have
$$
\begin{aligned}
{\Uparrow}_1&\defeq\{y\in X:xx\in X^1yX^1\}=\{y\in X:\exists a,b\in X^1\;\;ayb=xx\}\\
&=\{y\in X:\exists \{x_s\}_{s\in 2^{\le 1}},\{y_s\}_{s\in 2^{\le 1}}\subseteq X,\; \{a_s\}_{a\in 2^{\le 1}},\{b_s\}_{s\in 2^{\le 1}}\subseteq X^1,\\
&\hskip60pt x_{(0)}=x_{(1)}=x,\;y_{()}=x_{(0)}x_{(1)},\;x_{()}=y,\;y_{()}=a_{()}x_{()}b_{()}\}={\Uparrow}'_{\!1}x.
\end{aligned}
$$
Assume that for some $n\in\IN$ the equality ${\Uparrow}_{\!n}x={\Uparrow}'_{\!n}x$ has been proved. To check that ${\Uparrow}_{\!n{+}1}x\subseteq{\Uparrow}'_{\!n{+}1}x$, take any $x_{()}\in{\Uparrow}_{\!n{+}1}x$. The definition of ${\Uparrow}_{\!n{+}1}x$ ensures that $X^1x_{()}X^1\cap ({\Uparrow}_{\!n}x)^2\ne\emptyset$ and hence $a_{()}x_{()}b_{()}=x_{(0)}x_{(1)}$ for some $a_{()},b_{()}\in X^1$ and $x_{(0)}x_{(1)}\in{\Uparrow}_{\!n}x={\Uparrow}_{\!n}'x$.  By the definition of the set ${\Uparrow}'_{\!n}x$, for every $k\in\{0,1\}$, there exist sequences $\{x_{k\hat{\;}s}\}_{s\in 2^{\le n}},\{y_{k\hat{\;}s}\}_{s\in 2^{\le n}}\subseteq X$ and $\{a_{k\hat{\;}s}\}_{s\in 2^{\le n}},\{b_{k\hat{\;}s}\}_{s\in 2^{\le n}}\subseteq X^1$ such that 
\begin{itemize}
\item $x_{k\hat{\;}s}=x$ for all $s\in 2^n$;
\item $y_{k\hat{\;}s}=a_{k\hat{\;}s}x_{k\hat{\;}s}b_{k\hat{\;}s}$ for every $s\in 2^{\le n}$;
\item $y_{k\hat{\;}s}=x_{k\hat{\;}s\hat{\;}0}x_{k\hat{\;}s\hat{\;}1}$ for every $s\in 2^{<n}$.
\end{itemize}
Then the sequences $\{x_s\}_{s\in 2^{\le n{+}1}},\{x_s\}_{s\in 2^{\le n{+}1}}\subseteq X$ and $\{a_s\}_{s\in 2^{\le n{+}1}},\{b_s\}_{s\in 2^{\le n{+}1}}\subseteq X^1$ witness that $x_{()}\in{\Uparrow}'_{\!n{+}1}x$, which completes the proof of the inclusion ${\Uparrow}_{\!n{+}1}x\subseteq {\Uparrow}_{\!n{+}1}'x$.
\smallskip 

To prove that ${\Uparrow}_{\!n{+}1}'x\subseteq {\Uparrow}_{\!n{+}1}x$, take any $x_{()}\in  {\Uparrow}_{\!n{+}1}'x$ and by the definition of ${\Uparrow}_{\!n{+}1}'x$, find sequences $\{x_s\}_{s\in 2^{\le n{+}1}},\{x_s\}_{s\in 2^{\le n{+}1}}\subseteq X$ and $\{a_s\}_{s\in 2^{\le n{+}1}},\{b_s\}_{s\in 2^{\le n{+}1}}\subseteq X^1$ satisfying the conditions $(1_{n+1})$--$(3_{n+1})$. The for every $k\in\{0,1\}$ the sequences $\{x_{k\hat{\;}s}\}_{s\in 2^{\le n}},\{x_{k\hat{\;}s}\}_{s\in 2^{\le n}}\subseteq X$ and $\{a_{k\hat{\;}s}\}_{s\in 2^{\le n}},\{b_{k\hat{\;}s}\}_{s\in 2^{\le n}}\subseteq X^1$ witness that $x_{(0)},x_{(1)}\in {\Uparrow}'_{\!n}={\Uparrow}_{\!n}x$ and then the equalities $a_{()}x_{()}b_{()}=y_{()}=x_{(0)}x_{(1)}\in({\Uparrow}_{\!n}x)^2$ imply that  $X^1x_{()}X^1\cap({\Uparrow}_{\!n}x)^2\ne\emptyset$ and hence $x_{()}\in {\Uparrow}_{\!n{+}1}x$, which completes the proof of the equality $ {\Uparrow}_{\!n{+}1}x={\Uparrow}_{\!n{+}1}'x$.
\end{proof}

A semigroup $X$ is called {\em duo} if $aX=Xa$ for every $a\in X$. Observe that each commutative semigroup is duo.

The upper $\two$-classes in duo semigroups have the following simpler description.

\begin{theorem}\label{t:duo} For any element $a\in X$ of a duo semigroup $X$ we have 
$${\Uparrow}a=\{x\in X:a^\IN\cap XxX\ne\emptyset\}.$$
\end{theorem}

\begin{proof} First we prove that the set $\frac{a^\IN}X\defeq \{x\in X:a^\IN\cap XxX\ne\emptyset\}$ is contained in ${\Uparrow}a$. In the opposite case, we can find a point $x\in\frac{a^\IN}{X}\setminus {\Uparrow}a$. Taking into account that ${\Uparrow}a$ is a coideal containing $a$, we conclude that $a^\IN\subseteq{\Uparrow}a$ and $\emptyset= XxX\cap{\Uparrow}a\supseteq XxX\cap a^\IN$, which contradicts the choice of the point $x\in\frac{a^\IN}X$. This contradiction shows that $\frac{a^\IN}X\subseteq {\Uparrow}a$.

Next, we prove that $\frac{a^\IN}X$ is prime coideal. Since $X$ is a duo semigroup, for every $x\in X$ we have $X^1x=xX^1=X^1xX^1$. If $x,y\in \frac{a^\IN}X$, then $$X^1x\cap a^\IN=X^1xX^1\cap a^\IN\ne\emptyset\ne X^1yX^1\cap a^\IN=yX^1\cap a^\IN$$ and hence $X^1xyX^1\in a^\IN\ne\emptyset$, which means that $xy\in\frac{a^\IN}X$. Therefore, $\frac{a^\IN}X$ is a subsemigroup of $X$. The definition of $\frac{a^\IN}X$ ensures that $X\setminus \frac{a^\IN}X$ is an ideal in $X$. Then $\frac{a^\IN}X\subseteq{\Uparrow}a$ is a prime coideal in $X$ and $\frac{a^\IN}X={\Uparrow}a$, by the minimality of ${\Uparrow}a$.
\end{proof}

For viable semigroups Putcha and Weissglass \cite{PW} proved the following simplification of Proposition~\ref{p:Upclass}. Following Putcha and Weissglass  \cite{PW}, we define a semigroup $X$ to be {\em viable} if for any elements $x,y\in X$ with $\{xy,yx\}\subseteq E(X)$, we have $xy=yx$. For various equivalent conditions to the viability, see \cite{Ban}.

\begin{proposition}[Putcha--Weissglass]\label{p:PW} If $X$ is a viable semigroup, then for every idempotent $e\in E(X)$we have ${\Uparrow}e=\{x\in X:e\in X^1xX^1\}$.
\end{proposition}

\begin{proof} We present a short proof of this theorem, for convenience of the reader. Let ${\Uparrow}_{\!1}e\defeq \{x\in X:e\in X^1xX^1\}$. By Proposition~\ref{p:Upclass}, ${\Uparrow}_{\!1}e\subseteq {\Uparrow}e$. The reverse inclusion will follow from the minimality of the prime coideal ${\Uparrow}e$ as soon as we prove that ${\Uparrow}_{\!1}e$ is a prime coideal in $X$. It is clear from the definition that ${\Uparrow}_{\!1}e$ is a coideal. So, it remains to check that ${\Uparrow}_{\!1}e$ is a subsemigroup. Given any elements $x,y\in {\Uparrow}_{\!1}e$, find elements $a,b,c,d\in X^1$ such that $axb=e=cyd$. Then $axbe=ee=e$ and $(beax)(beax)=be(axbe)ax=beeax=beax$, which means that $beax$ is an idempotent. By the viability of $X$, $axbe=e=beax$. By analogy we can prove that $ecyd=e=ydec$. Then $aeaxydex=ee=e$ and hence $xy\in{\Uparrow}_{\!1}e$.
\end{proof}

Proposition~\ref{p:PW} has an important corollary, proved in \cite{PW}.

\begin{corollary}[Putcha--Wiessglass] If $X$ is a viable semigroup, then for every $x\in X$ its $\two$-class ${\Updownarrow}x$ contains at most one idempotent.
\end{corollary}

\begin{proof} To derive a contradiction, assume that the semigroup ${\Updownarrow}x$ contains two distinct idempotents $e,f$. By Proposition~\ref{p:PW}, there are elements $a,b,c,d\in X^1$ such that $e=afb$ and $f=ced$. Observe that $afbe=ee=e$ and $(beaf)(beaf)=be(afbe)af=beeaf=beaf$ and hence $afbe$ and $beaf$ are idempotents. The viability of $X$ ensures that $afbe=beaf$. By analogy we can prove that $eafb=e=efbea$, $cedf=f=dfce$ and $fced=f=edfc$. These equalities imply that $H_e=H_f$ and hence $e=f$ because the group $H_e=H_f$ contains a unique idempotent. But the equality $e=f$ contradicts the choice of the idempotents $e,f$.
\end{proof}

\section{The structure of $\two$-trivial semigroups}

Tamura's Theorem~\ref{t:Tamura} motivates the problem of a deeper study of the structure of $\two$-trivial semigroups. This problem has been considered in the literature, see, e.g. \cite[\S3]{Petrich64}.  Proposition~\ref{p:smallest-pi} implies the following simple characterization of $\two$-trivial semigroups.

\begin{theorem}\label{t:primesimple} A semigroup $X$ is $\two$-trivial if and only if every nonempty prime ideal in $X$ coincides with $X$.
\end{theorem}

Observe that a semigroup $X$ is $\two$-trivial if and only if $X={\Uparrow}x$ for every $x\in X$. This observation and Propositions~\ref{p:Upclass} and \ref{p:Upclass-tree} imply the following characterization.

\begin{proposition} A semigroup $X$ is $\two$-trivial if and only if for every $x,y\in X$ there exists $n\in\IN$ and  sequences $\{a_{s}\}_{s\in 2^{\le n}},\{b_s\}_{s\in 2^{\le n}}\subseteq X^1$ and $\{x_s\}_{s\in 2^{\le n}},\{y_s\}_{s\in 2^{\le n}}\subseteq X$ satisfying the following conditions:
\begin{enumerate}
\item $x_s=x$ for all $s\in 2^n$;
\item $y_s=a_sx_sb_s$ for every $s\in 2^{\le n}$;
\item $y_s=x_{s\hat{\;}0}x_{s\hat{\;}1}$ for every $s\in 2^{<n}$;
\item $x_{()}=y$ for the unique element $()$ of $2^0$;
\end{enumerate}
\end{proposition}

A semigroup $X$ is called {\em Archimedean} if for any elements $x,y\in X$ there exists $n\in\IN$ such that $x^n\in XyX$ for some $a,b\in X$. A standard example of an Archimedean semigroup is the additive semigroup $\IN$ of positive integers. For commutative semigroups the following characterization was obtained by Tamura and Kimura in \cite{TK54}.

\begin{theorem}\label{t:Archimed} A duo semigroup $X$ is $\two$-trivial if and only if $X$ is Archimedean.
\end{theorem}

\begin{proof} If $X$ is $\two$-trivial, then by Theorem~\ref{t:duo}, for every $x,y\in X$ there exists $n\in\w$ such that $x^n\in XyX$, which means that $X$ is Archimedean.

If $X$ is Archimedean, then for every $\in X$, we have
$${\Uparrow}x=\{y\in X:x^\IN\cap (XxX)\ne\emptyset\}=X,$$  see Theorem~\ref{t:duo}, which means that the semigroup $X$ is $\two$-trivial.
\end{proof}

Following Tamura \cite{Tamura82}, we define a semigroup $X$ to be {\em unipotent} if $X$ contains a unique idempotent.

\begin{theorem}[Tamura, 1982]\label{t:max-ideal} For the unique idempotent $e$ of an unipotent $\two$-trivial semigroup $X$, the maximal group $H_e$ of $e$ in $X$ is an ideal in $X$. 
\end{theorem}

\begin{proof} This theorem was proved by Tamura in \cite{Tamura82}. We present here an alternative (and direct) proof. To derive a contradiction, assume that $H_e$ is not an ideal in $X$. Then the set $I\defeq\{x\in X:\{ex,xe\}\not\subseteq H_e\}$ is not empty. We claim that $I$ is an ideal in $X$. Assuming the opposite, we could find $x\in I$ and $y\in X$ such that $xy\notin I$ or $yx\notin I$.

If $xy\notin I$, then $\{exy,xye\}\subseteq H_e$. Taking into account that $exy$ and $xye$ are elements of the group $H_e$, we conclude that $exy=exye=xye$. 
Let $g$ be the inverse element to $xye$ in the group $H_e$. Then $exyg=xyeg=xyg=e$. Replacing $y$ by $yg$, we can assume that $ye=y$ and $xy=e$. Observe that $yxyx=y(xy)x=yex=(ye)x=yx$, which means that $yx$ is an idempotent in $S$. Since $e$ is a unique idempotent of the semigroup $X$, $yx=e=xy$. It follows that $xe=x(yx)=(xy)x=ex$ and $ey=(yx)y=y(xy)=ye=y$. Using this information it is easy to show that $xe=ex\in H_e$. By analogy we can show that the assumption $yx\notin I$ implies $ex=xe\in H_e$. So, in both cases we obtain $ex=xe\in H_e$,  which contradicts the choice of $x\in I$.

This contradiction shows that $I$ is an ideal in $S$. Observe that for any $x,y\in X\setminus I$ we have $\{ex,xe,ey,ye\}\subseteq H_e$. Then also $xye=x(eye)=(xe)(ye)\in H_e$ and $exy=(exe)y=(ex)(ey)\in H_e$, which means that $xy\in X\setminus I$ and hence $I$ is a nontrivial prime ideal in $X$. But the existence of such an ideal contradicts the $\two$-triviality of $X$.
 \end{proof}
 
 An element $z$ of a semigroup $X$ is called {\em central} if $zx=xz$ for all $x\in X$.
 
\begin{corollary}\label{c:EZK}The unique idempotent $e$ of a unipotent $\two$-trivial semigroup $X$ is central in $X$.
\end{corollary}

\begin{proof} Let $e$ be a unique idempotent of the unipotent semigroup $X$. By Tamura's Theorem~\ref{t:max-ideal}, the maximal subgroup $H_e$ of $e$ is an ideal in $X$. Then for every $x\in X$ we have $xe,ex\in H_e$. Taking into account that $xe$ and $ex$ are elements of the group $H_e$, we conclude that $ex=exe=xe$. This means that the idempotent $e$ is central in $X$. 
\end{proof}

As we already know a semigroup $X$ is $\two$-trivial if and only if each nonempty prime ideal in $X$ is equal to $X$. 

A semigroup $X$ is called
\begin{itemize}
\item {\em simple} if every nonempty ideal in $X$ is equal to $X$;
\item {\em $0$-simple} if contains zero element $0$, $XX\ne\{0\}$ and every nonempty ideal in $X$ is equal to $X$ or $\{0\}$; 
\item {\em congruence-free} if every congruence on $X$ is equal to $X\times X$ or $\Delta_X\defeq\{(x,y)\in X\times X:x=y\}$.
\end{itemize}

It is clear that a semigroup $X$ is $\two$-trivial if $X$ is either simple or congruence-free. On the other hand the additive semigroup of integers $\mathbb N$ is $\two$-trivial but not simple.

\begin{remark}\label{r:Bard} By \cite{ACMU}, \cite{CM}, there exists an infinite $0$-simple congruence-free monoid $X$. Being congruence-free, the semigroup $X$ is $\two$-trivial. On the other hand, $X$ contains at least two central idempotents: $0$ and $1$. The $\two$-trivial monoid $X$ is not unipotent and its center $Z(X)=\{z\in X:\forall x\in X\;(xz=zx)\}$ is not $\two$-trivial. The polycyclic monoids (see \cite{BG1}, \cite{BG2}, \cite{Bard16}, \cite{Bard20}) have the similar properties. By Theorem 2.4 in \cite{BG1}, for $\lambda\ge 2$ the polycyclic monoid $P_\lambda$ is congruence-free and hence $\two$-trivial, but its center $Z(P_\lambda)=\{0,1\}$ is not $\two$-trivial.  
\end{remark}

\section{Acknowledgements} The authors express their sincere thanks to Oleg Gutik and Serhii Bardyla for valuable information on congruence-free monoids (see Remark~\ref{r:Bard}) and to all listeners of Lviv Seminar in Topological Algebra (especially, Alex Ravsky) for active listening of the talk of the first named author that allowed to notice and then correct a crucial gap in the initial version of this manuscript.

\end{document}